\documentclass[12pt]{amsart}
\usepackage{graphics,wrapfig, graphicx, mathrsfs,xcolor}
\usepackage{mathtools}
\usepackage{amssymb,amscd,graphicx,xcolor,subfig,tikz}
\usepackage{graphics}
\usepackage{indentfirst}
\usepackage{bm, enumerate,xcolor}
\usepackage[dvips]{epsfig}
\usepackage{latexsym}
\usepackage{float}

\usepackage{graphics,url,color,epsfig}
\usepackage{tikz}
\usepackage{xcolor}
\definecolor{mysoftblue}{hsb}{0.55,0.15,0.9}
\definecolor{mysoftpurple}{hsb}{0.9,0.25,0.9}
\definecolor{mysoftorange}{hsb}{0.15,0.3,0.95}
\usetikzlibrary{patterns}

\usepackage[colorlinks]{hyperref}
\AtBeginDocument{
   \hypersetup{
    linkcolor=blue,
    citecolor=blue,
 }
}

\usepackage{amsmath}
\usepackage[applemac]{inputenc}
\usepackage[T1]{fontenc}
\newtheorem{lemma}{Lemma}[section]
\newtheorem{theorem}{Theorem}[section]

\newtheorem{proposition}{Proposition}[section]
\newtheorem{conjecture}{Conjecture}[section]

\newtheorem{definition}{Definition}[section]
\newtheorem{remark}{Remark}[section]

\numberwithin{equation}{section} \numberwithin{theorem}{section}
\numberwithin{example}{section} \numberwithin{remark}{section}
\numberwithin{figure}{section} \numberwithin{algorithm}{section}
\mathtoolsset{showonlyrefs}

\title[Green function rigidity]{Green function rigidity for two dimensional sphere}

\author{Mijia Lai}
\address{School of Mathematical Sciences, Shanghai Jiao Tong University, Shanghai 200240}\email{laimijia@sjtu.edu.cn}
\author{Chilin Zhang}
\address{School of Mathematical Sciences, Fudan University, Shanghai 200433, China}\email{zhangchilin@fudan.edu.cn}

\begin{document}
\begin{abstract}
We verify a conjecture proposed by X. Chen and Y. Shi, which arises from their study of the Green function on spheres in Euclidean space. More precisely, let $M\subset \mathbb{R}^3$ be a closed $C^{2}$ embedded surface and suppose that there exists a point $p\in M$ so that its Green function $G$ is of the form $G(p,q)=-\frac{1}{2\pi} \ln d_{\mathbb{R}^3}(p,q)+c, \forall q\neq p$, then $M$ must be a round sphere. 
\end{abstract}
\maketitle

\section{Introduction} 

In \cite{CS1}, the authors made interesting investigations on the Green function of various conformal invariant operators on the round sphere in the Euclidean space.  An interesting finding is that the explicit expression of the Green function is written in terms of the extrinsic distance of the underlying Euclidean space. Denote by $||\cdot||$ the Euclidean norm of a vector, then we have the following. 
\begin{itemize}
    \item For $\mathbb{S}^2\subset\mathbb{R}^3$, the Green function of $\Delta_g$ has the form $G(p, q)=-\frac{1}{2 \pi} \log \|p-q\|+C$;
    \item For $\mathbb{S}^n\subset  \mathbb{R}^{n+1}, n\geq 3$, the Green function for the conformal Laplacian $P_2^g$ has the form $G(p, q)=c_{n, 1}\|p-q\|^{2-n}$;
    \item For $\mathbb{S}^4\subset  \mathbb{R}^5$, the Green function for the Paneitz operator $P_4^g$ has the form $G(p, q)= -\frac{1}{8 \pi^2} \log \|p-q\|+C$;
    \item For $\mathbb{S}^n\subset \mathbb{R}^{n+1}, n\geq 3, n\neq 4$, the Green function for the Paneitz operator $P_4^g$ has the form $G(p, q)=c_{n, 2}\|p-q\|^{4-n}$.
\end{itemize}

The authors then asked the converse question: let $M\subset \mathbb{R}^n$ be a closed embedded hypersurface, suppose that its Green function has the form as above, is it true that $M$ must be isometric to the round sphere?  It can be shown that if the particular form of Green function holds for any pair of points $(p,q)$ on $M$, then all points on $M$ are umbilical, and thus the answer to the converse question is certainly yes. In view of this, Chen-Shi raised a more difficult question, referred to as the Strong Green function rigidity conjecture: 

\begin{conjecture}
    Let $M^n \subset \mathbb{R}^{n+1}$ be a smooth closed embedded hypersurface, suppose that there exists a point $p\in M$, such that its Green function of the corresponding conformal invariant operator is of the form as above, then $M$ must be a round sphere. 
\end{conjecture}

Using the positive mass theorem, Chen-Shi, Chen-Gan-Shi verified their conjecture for the conformal Laplacian in dimension $n\geq 3$ \cite{CS1,CGS,CS2}. The key ingredient of their proof is the connection with asymptotic flat manifold. Indeed, using $G^{\frac{4}{n-2}}(p, \cdot)$ as a conformal factor, one can  turn $M$ into an asymptotic flat manifold with $p$ being opened up to the infinity. The particular form of the Green function enables one to compute the ADM mass, which happens to be zero. Consequently, the rigidity conclusion follows. 

However, such an approach does not appear to be feasible in two dimensions. In this paper, we adopt an alternative strategy and provide a positive resolution to the Strong Green Function Rigidity Conjecture for the standard Laplace on two sphere.

\begin{theorem}\label{Tmain}
    Let $M\subset \mathbb{R}^3$ be a $C^2$ closed embedded surface, suppose that there exists a point $p\in M$ and its Green function $G$ is of the form 
    \[
    G(p,q)=-\frac{1}{2\pi} \ln ||p-q||+c, \quad \forall q\in M\setminus p,
    \]
    then $M$ must be a round sphere. 
\end{theorem}

The main idea of the proof is as follows. First, we show that the particular form of Green function is equivalent to an equation (see \eqref{eq:surface2}) for the mean curvature of $M$. Secondly, based on \eqref{eq:surface2}, $M$ is star-shaped and thus we can view it as a radial graph over $\mathbb{S}^2_{+}$. Consequently we turn \eqref{eq:surface2} into an equation (see \eqref{eq: radial}) that the radial function $\rho$ satisfies. This is a minimal surface type equation on $\mathbb{S}^2_{+}$ with Dirichlet boundary condition being $\rho=-\infty$. Similar equations for the radial graph have been a subject of continuing research, see, e.g. \cite{FR, L, S, T}. Finally, we use the moving plane (indeed, rotating circle) method to show that any solution of \eqref{eq: radial} must be rotational symmetric, from which the main conclusion follows. 

We also remark that Chen-Shi\cite{CS1} has verified the strong Green function conjecture for two sphere under assumptions that $M$ is analytic and rotational symmetric. 

The organization of the paper is as follows. In Section 2, we setup the basic notation and derive several equivalent equations based on the particular form of the Green function. We also establish several lemmas to deal with \eqref{eq: radial}. In Section 3, we provide detailed proof of theorem~\ref{Tmain}. In Section 4, borrowing some ideas from Chen-Shi's work, we make further discussions on our method.

\section{Preparations}
\subsection{Basic setup}
Let $M\subset \mathbb{R}^3$ be a compact embedded surface. 
Let $\nu$ be the exterior unit normal vector field on $M$. The second fundamental form is given by $II(e_i, e_j)=\langle \nabla_{e_i} \nu, e_j\rangle$, 
we use $H=\kappa_1+\kappa_2$ for mean curvature. 
In this notation, the unit two sphere has $H=2$ with respect to the unit outward normal. 

Under the assumptions of Theorem~\ref{Tmain}, the Green function  of $M$ is of the form 
$G(p,q)=-\frac{1}{2\pi} \log ||p-q||+c$, for some fixed $p$ and any $q\neq p$. 
We may set $p$ at the origin and simply write Green function as 
\[
G(y)=-\frac{1}{2\pi} \log |y|+c. 
\]
As rescaling does not change the Green function, we also assume that 
$\operatorname{Area}(M)=4\pi$. 

Hence, on $M$ we have  
\begin{align} \label{eq:laplaceonM}
\Delta_{M} G=\frac{1}{4\pi}- \delta_0
\end{align}
in the sense of distribution.
A direct computation shows that 
\begin{align}  \label{eq:lapalce}
\Delta_{\mathbb{R}^3} G(y) =-\frac{1}{2\pi } \frac{1}{|y|},\quad y\neq 0.
\end{align}
We also have 
\begin{align} \label{eq:surface1}
\Delta_{\mathbb{R}^3} G(y) =\nabla^2 G(\nu, \nu) + H \frac{\partial G}{\partial \nu} +\Delta_{M} G(y),
\end{align}
Since $\nabla^2 G=-\frac{1}{2\pi} \left(\frac{\delta_{ij}}{|y|^2}-\frac{2y_{i}y_{j}}{|y|^4}\right)$ given that $G(y)=-\frac{1}{2\pi} \log |y|+c$, 
combining \eqref{eq:laplaceonM}, \eqref{eq:lapalce}, \eqref{eq:surface1}, we obtain 
\begin{align} \label{eq:surface2}
    \frac{2(\langle y, \nu \rangle)^2}{|y|^4}- H \frac{\langle y,\nu \rangle }{|y|^2}+\frac{1}{2}=0 \quad, x\in M\setminus\{0\}.
\end{align}
Observe that $\langle y, \nu \rangle\neq0$ in $M\setminus\{0\}$, i.e., $M$ (more precisely, the region enclosed by $M$) is star-shaped with respect to the origin.

It follows from \eqref{eq:surface2} that $\langle y, \nu \rangle\neq 0$ for all $y\neq 0$. Then, for $y\neq 0$, \eqref{eq:surface2} reduces to  
\[
H=\frac{2\langle y, \nu\rangle}{|y|^2}+\frac{|y|^2}{2\langle y,\nu\rangle}\geq 2. 
\]

We can write $M$ near the origin as a graph of function $y_3=f(y_1,y_2)$ with $\nabla f(0)=0$. It follows that 
\[
\nu=\frac{(-f_1,-f_2, 1)}{\sqrt{1+|D f|^2}}.
\]
Using Taylor expansion at $(y_1, y_2)=(0,0)$, we find that
\[
\langle y, \nu\rangle=-\frac{1}{2}f_{11}(0,0) y_1^2-f_{12}(0,0)y_1y_2-f_{22}(0,0) y_2^2+o(r^2), \quad \text{ where $r^2=y_1^2+y_2^2$}.
\]
We also have $|y|^2=y_1^2+y_2^2+y_3^2=y_1^2+y_2^2+o(r^2)$. Then, it follows that 
\[
\lim_{(y_1,y_2)=\lambda e_i, \lambda\to 0} \frac{\langle y, \nu\rangle}{|y|^2}=\frac{1}{2}\kappa_i,  \quad i=1,2
\]
where $e_i$ are the principal directions of $M$ at the origin with principal curvatures $\kappa_i$. Plugging back to \eqref{eq:surface2}, we infer that $\kappa_1=\kappa_2=1$ and consequently the origin is an \textbf{umbilic} point of $M$.  

\subsection{Radial graph}
As observed above, $M$ is star-shaped with respect to the origin. Moreover $M$ is locally convex near the origin, thus we can express it as a radial graph over the upper hemisphere $\mathbb{S}^2_{+}$, i.e., 
$X(q)=e^{\rho(q)} q$, $q\in \mathbb{S}^2_{+}$.  
Let $\{e_1, e_2\}$ be a local orthonormal frame on $\mathbb{S}^{2}_{+}$, 
we can write the induced metric on $M$ as 
\[
g_{ij}=e^{2\rho}(\delta_{ij}+\rho_i\rho_j).
\]
Its inverse is 
\[
g^{ij}=e^{-2\rho}(\delta_{ij}-\frac{\rho_i\rho_j}{1+|\nabla \rho|^2}).
\]
The outward unit normal is given by 
\[
\nu(q)=\frac{q-\nabla \rho(q)}{\sqrt{1+|\nabla \rho|^2}}.
\]
The second fundamental form is (write $\nabla$ as the covariant derivative)
\[
A_{ij}=-\langle X_{ij}, \nu\rangle=\frac{e^\rho\left(\delta_{i j}+ \nabla_i \rho \nabla_j \rho-\nabla_{i j} \rho\right)}{\sqrt{1+|\nabla \rho|^2}} .
\]
Then
\[A_{i}^{j}=A_{ik}g^{kj}=\frac{e^{-\rho}\delta_{i}^{j}}{\sqrt{1+|\nabla\rho|^{2}}}-e^{-\rho}\partial_{i}(\frac{\rho_{j}}{\sqrt{1+|\nabla\rho|^{2}}}),\]
hence 
\[
H=g^{ij} A_{ji}=e^{-\rho}\left( \frac{2}{\sqrt{1+|\nabla\rho|^2}}-\operatorname{div}(\frac{\nabla \rho}{\sqrt{1+|\nabla \rho|^2}})\right).
\]
Then \eqref{eq:surface1} becomes 
\begin{align} \label{eq: radial}
\operatorname{div}(\frac{\nabla \rho}{\sqrt{1+|\nabla \rho|^2}})=-\frac{1}{2} e^{2\rho}\sqrt{1+|\nabla \rho|^2}.
\end{align}

Again if we express $y\in M$ near the origin as the graph of a function $y_3=f(y_1,y_2)$ with $\nabla f(0)=0$. Let $x=\frac{y}{|y|}$ be the projection on $\mathbb{S}^{2}_{+}$, then $e^{\rho(x)}=|y|$. Let
\begin{equation*}
    (r,\phi)=\Big(\sqrt{y_{1}^{2}+y_{2}^{2}},\angle(y_{1}\vec{e}_{1}+y_{2}\vec{e}_{2},\vec{e}_{1})\Big)
\end{equation*}
be the polar coordinate in the $y_{1}-y_{2}$ plane, we then have 
\[
r^2+f(y_{1},y_{2})^2= e^{2\rho(x)}. 
\]
Since the origin is an umbilic point with principal curvature $1$, then as $(y_{1},y_{2})\to0$,
\begin{equation*}
    f(y_{1},y_{2})=\frac{1}{2} r^2+\mathcal{E}_{1}(y_{1},y_{2})=\frac{1}{2} r^2+o(r^2)\quad\mbox{with }D^{2}_{(y_{1},y_{2})}\mathcal{E}_{1}=o(1).
\end{equation*}
Then, in the polar coordinate $(r,\phi)$, we have
\begin{equation*}
    D_{(r,\phi)}\mathcal{E}_{1}=\begin{bmatrix}o(r)&o(r^{2})\end{bmatrix},\quad D^{2}_{(r,\phi)}\mathcal{E}_{1}=\begin{bmatrix}o(1)&o(r)\\o(r)&o(r^{2})\end{bmatrix}.
\end{equation*}
We apply the inverse function theorem, precisely treating $(f,\phi)$ as a new set of variables and writing $r$ as a function of $(f,\phi)$, then
\begin{equation*}
    \frac{\partial r}{\partial f}=\frac{1}{r+o(r)}=\frac{1+o(1)}{r},\quad\frac{\partial r}{\partial\phi}=\frac{-o(r^{2})}{r+o(r)}=o(r),
\end{equation*}
which means the following holds for $\mathcal{E}_{2}(f,\phi)=-2\mathcal{E}_{1}(r,\phi)$:
\[
r(f,\phi)^2=2f+\mathcal{E}_{2}(f,\theta)=2 f +o(f),\quad D_{(f,\phi)}\mathcal{E}_{2}=\begin{bmatrix}o(1)&o(f)\end{bmatrix}.
\]
Using $r^{2}+f^{2}=e^{2\rho}$, we have for $\mathcal{E}_{3}(f,\phi)=\mathcal{E}_{2}(f,\phi)+f^{2}$, the following holds:
\[
e^{2\rho(x)}=2 f +\mathcal{E}_{3}(f,\phi)=2 f +o(f),\quad D_{(f,\phi)}\mathcal{E}_{3}=\begin{bmatrix}o(1)&o(f)\end{bmatrix}.
\]
Also note that $f=e^{\rho(x)}\cdot \langle x, \vec{e_3}\rangle$ and $\phi=\angle(x_{1}\vec{e}_{1}+x_{2}\vec{e}_{2},\vec{e}_{1})$, thus we have 
\begin{align}
     \label{asym1}
e^{\rho(x)}&=2 \langle x, \vec{e}_3 \rangle +o(\langle x, \vec{e}_3\rangle),\quad \text{as } \langle x, \vec{e}_3\rangle\to 0,
\end{align}
and
\begin{align} \label{asym2}
e^{\rho} \nabla \rho &= 2 \vec{e}_3 +o(1),                              \quad  \text{as } \langle x, \vec{e}_3\rangle \to 0.
\end{align}

\subsection{Two algebraic lemmas}
To deal with minimal surface type equation \eqref{eq: radial}, we prove two algebraic lemmas. 
\begin{lemma}\label{lem. uniform elliptic in energy sense}
    Assume that $\vec{a},\vec{b}\in\mathbb{R}^{2}$ with $|\vec{a}|,|\vec{b}|\leq M$. Then, there exists some $\lambda>1$ depending only on $M$ such that
    \begin{equation*}
        \Big|\frac{\vec{b}}{\sqrt{1+|\vec{b}|^{2}}}-\frac{\vec{a}}{\sqrt{1+|\vec{a}|^{2}}}\Big|\leq\lambda|\vec{b}-\vec{a}|,
    \end{equation*}
    and
    \begin{equation*}
        \langle\frac{\vec{b}}{\sqrt{1+|\vec{b}|^{2}}}-\frac{\vec{a}}{\sqrt{1+|\vec{a}|^{2}}},\vec{b}-\vec{a}\rangle\geq\lambda^{-1}|\vec{b}-\vec{a}|^{2}.
    \end{equation*}
\end{lemma}
\begin{proof}
    We notice that for $\vec{v}=x\cdot\vec{e}_{1}+y\cdot\vec{e}_{2}$, $\displaystyle\frac{\vec{v}}{\sqrt{1+|\vec{v}|^{2}}}$ is actually the gradient of $f(x,y)=\sqrt{1+x^{2}+y^{2}}$. We can easily verify the uniform convexity of $f$, that is:
    \begin{equation}\label{eq. uniform convex}
        \lambda(M)^{-1}I_{2}\leq D^{2}f(\vec{v})\leq\lambda(M)I_{2},\quad\mbox{as long as }|\vec{v}|\leq M.
    \end{equation}
    For $|\vec{a}|,|\vec{b}|\leq M$ with $L=|\vec{b}-\vec{a}|$, we consider their convex combination curve:
    \begin{equation*}
        \gamma(t)=(1-t)\vec{a}+t\vec{b}.
    \end{equation*}
    Then, $|\gamma(t)|\leq M$, $\gamma'(t)=\vec{b}-\vec{a}$, and
    \begin{equation*}
        \frac{\vec{b}}{\sqrt{1+|\vec{b}|^{2}}}-\frac{\vec{a}}{\sqrt{1+|\vec{a}|^{2}}}=\int_{0}^{1}D^{2}f(\gamma(t))\cdot\gamma'(t)dt=\int_{0}^{1}D^{2}f(\gamma(t))\cdot(\vec{b}-\vec{a})dt.
    \end{equation*}
    By the uniform convexity of $f$, we can prove the desired two inequalities.
\end{proof}
\begin{lemma}\label{lem. linear algebra}
    Assume that $\vec{a},\vec{b}\in T_{p}\mathbb{S}^{2}_{+}$, such that $|\vec{b}|\leq\lambda|\vec{a}|$ and $\langle\vec{b},\vec{a}\rangle\geq\lambda^{-1}|\vec{a}|^{2}$ for some $\lambda>1$. Then, there exists a self adjoint tensor $A\in T_{1}^{1}(T_{p}\mathbb{S}^{2}_{+})$, such that $\vec{b}=A\vec{a}$, and
    \begin{equation*}
        \{\Lambda^{-1}g_{ij}\}\leq\{A_{i}^{k}g_{kj}\}\leq\{\Lambda g_{ij}\}
    \end{equation*}
    for some $\Lambda>1$ depending only on $\lambda$.
\end{lemma}
\begin{proof}
    Without loss of generality, we assume that $T_{p}\mathbb{S}^{2}_{+}=\mbox{span}(\vec{e}_{1},\vec{e}_{2})$ and that $\vec{a}=\vec{e}_{1}$. Then, we have $\vec{b}=b^{i}\vec{e}_{i}$ with $b^{1}\geq\lambda^{-1}$ and $|\vec{b}|\leq\lambda$. Let $(\vec{\eta}^{1},\vec{\eta}^{2})$ be a co-frame dual to $(\vec{e}_{1},\vec{e}_{2})$, then we let
    \begin{equation*}        A=b^{1}\vec{\eta}^{1}\otimes\vec{e}_{1}+b^{2}\vec{\eta}^{1}\otimes\vec{e}_{2}+b^{2}\vec{\eta}^{2}\otimes\vec{e}_{1}+3\lambda\vec{\eta}^{2}\otimes\vec{e}_{2}.
    \end{equation*}
    One can easily see that $A$ is self-adjoint with $\vec{b}=A\vec{a}$. Moreover, we have for $g_{ij}=\delta_{ij}$
    \begin{equation*}
        \{\frac{\lambda^{-1}}{2}\delta_{ij}\}\leq\{A_{i}^{k}\delta_{kj}\}\leq\{4\lambda\delta_{ij}\}.
    \end{equation*}
    Therefore, we can choose $\Lambda=4\lambda$.
\end{proof}

\section{Proof of Theorem~\ref{Tmain}}

\begin{proof}[Proof of Theorem~\ref{Tmain}]
As indicated in the Introduction, we shall use the method of moving plane (rotating circle to be more precise) to prove the rotational symmetry of the solution $\rho(x)$ of the following Dirichlet problem on $\mathbb{S}^2_{+}$:
\begin{align}  
\left\{
  \begin{array}{ll}
    \operatorname{div}(\frac{\nabla \rho(x)}{\sqrt{1+|\nabla \rho(x)|^2}})=-\frac{1}{2} e^{2\rho(x)}\sqrt{1+|\nabla \rho(x)|^2} , &  \quad x\in 
    \mathbb{S}^2_{+}, \\ 
    e^{\rho(x)}=0, & \quad x\in \partial \mathbb{S}^2_{+},
  \end{array}
\right.
\end{align}
with additional asymptotic assumptions \eqref{asym1} and \eqref{asym2}.

{\bf 1. Setup of the rotating circles}

\begin{definition}\label{def. moving plane setting}
    Let $\theta\in(0,\frac{\pi}{2})$ and let $\vec{v}_{\theta}=\cos{\theta}\vec{e}_{3}-\sin{\theta}\vec{e}_{1}$. Let
    \begin{equation*}
        T_{\theta}=\{x\in\mathbb{S}^{2}_{+}:\ \langle x,\vec{v}_{\theta}\rangle=0\},\quad\Sigma_{\theta}=\{x\in\mathbb{S}^{2}_{+}:\ \langle x,\vec{v}_{\theta}\rangle<0\}.
    \end{equation*}
$T_{\theta}$ is the half circle that we are going to rotate. 
    Let $x_{\theta}\in\Sigma_{2\theta}\setminus\Sigma_{\theta}$ be the reflection of $x$ about $T_{\theta}$, namely:
    \begin{equation*}
        x_{\theta}=x-2\langle x,\vec{v}_{\theta}\rangle\vec{v}_{\theta}.
    \end{equation*}
    Let $\rho_{\theta}(x)=\rho(x_{\theta})$ and $w_{\theta}(x)=\rho_{\theta}(x)-\rho(x)$ for $x\in\Sigma_{\theta}$. For every $\sigma>0$, we denote:
    \begin{equation*}
        w_{\theta,\sigma}=w_{\theta}+\max\Big\{\frac{\sigma}{\langle x,\vec{e}_{3}\rangle}-1,0\Big\},\quad\Omega_{\theta,\sigma}=\{x\in\Sigma_{\theta}:w_{\theta,\sigma}(x)<0\}.
    \end{equation*}
\end{definition}

\begin{figure}[htbp]
    \centering
    \includegraphics[width=0.5\linewidth]{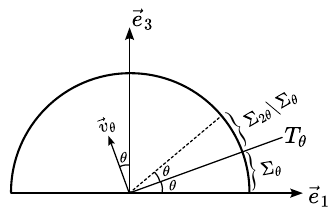}
\end{figure}

{\bf 2. Initiating the moving circle}
\begin{lemma}\label{lem. strong inclusion}
    For every $\theta\in(0,\frac{\pi}{2})$, there exists some $\overline{\sigma}>0$ (depending on $\theta$), such that for every $\sigma\in(0,\overline{\sigma})$, we have
    \begin{equation}\label{eq. boundary is non-negative}
        \liminf_{x\to y}w_{\theta,\sigma}(x)\geq0,\quad\mbox{for every }y\in\partial\Sigma_{\theta}.
    \end{equation}
    Besides, the distance between $\Omega_{\theta,\sigma}$ and $\partial\mathbb{S}^{2}_{+}$ is strictly positive.
\end{lemma}
\begin{proof}
    In fact, as $\langle x,\vec{e}_{3}\rangle\to0$, the following uniform estimate holds:
    \begin{equation*}
        |\rho_{\theta}(x)|,|\rho(x)|=O(|\ln{\langle x,\vec{e}_{3}\rangle}|)=o(\frac{1}{\langle x,\vec{e}_{3}\rangle}).
    \end{equation*}
    Therefore, $\Omega_{\theta,\sigma}\subseteq\Sigma_{\theta}\cap\{\langle x,\vec{e}_{3}\rangle\geq h(\sigma)\}$. In particular, we have shown \eqref{eq. boundary is non-negative} if $y\in\partial\mathbb{S}^{2}_{+}$. Besides, as $x_{\theta}=x$ on $T_{\theta}$, we have
    \begin{equation*}
        w_{\theta,\sigma}=\max\Big\{\frac{\sigma}{\langle x,\vec{e}_{3}\rangle}-1,0\Big\}\geq0\mbox{ on }T_{\theta}.
    \end{equation*}
    As we can easily verify that $w_{\theta,\sigma}$ has a bounded gradient in the region $\Sigma_{\theta}\cap\{\langle x,\vec{e}_{3}\rangle\geq h(\sigma)\}$, this implies \eqref{eq. boundary is non-negative} when $y\in T_{\theta}$.
\end{proof}
\begin{lemma} \label{L1}
    There exists some $\epsilon>0$, such that for all $\theta<\epsilon$, we have $w_{\theta}(x)\geq0$ in $\Sigma_{\theta}$.
\end{lemma}
\begin{remark}
    In fact, we even have that $w_{\theta}(x)>0$ in $\Sigma_{\theta}$ by the strong maximum principle. Such a strict inequality will be proven soon in the proof of Lemma~\ref{L2}.
\end{remark}
\begin{proof}[Proof of Lemma~\ref{L1}]
    It suffices to show $w_{\theta,\sigma}\geq0$ for any $\sigma\in(0, \bar{\sigma})$. Suppose that this is not true for some $\sigma>0$, then $\Omega_{\theta,\sigma}\neq\emptyset$. We multiply on both sides of
    \begin{equation*}
        \mathrm{div}(\frac{\nabla\rho}{\sqrt{1+|\nabla\rho|^{2}}})-\mathrm{div}(\frac{\nabla\rho_{\theta}}{\sqrt{1+|\nabla\rho_{\theta}|^{2}}})=\frac{e^{2\rho_{\theta}}}{2}\sqrt{1+|\nabla\rho_{\theta}|^{2}}-\frac{e^{2\rho}}{2}\sqrt{1+|\nabla\rho|^{2}}=O(1),
    \end{equation*}
    by $w_{\theta,\sigma}$ and integrate on $\Omega_{\theta,\sigma}$. By Lemma~\ref{lem. strong inclusion}, no boundary term is produced when we integrate by parts, so we obtain
    \begin{equation}\label{eq. energy inequality for small theta}
        \int_{\Omega_{\theta,\sigma}}\langle\nabla w_{\theta,\sigma},\frac{\nabla\rho_{\theta}}{\sqrt{1+|\nabla\rho_{\theta}|^{2}}}-\frac{\nabla\rho}{\sqrt{1+|\nabla\rho|^{2}}}\rangle dx\leq C\int_{\Omega_{\theta,\sigma}}|w_{\theta,\sigma}|dx.
    \end{equation}
    Let $\epsilon$ be sufficiently small. For $x\in\Sigma_{\theta}$, we have
    \begin{equation*}
        \nabla\max\Big\{\frac{\sigma}{\langle x,\vec{e}_{3}\rangle}-1,0\Big\}=-\frac{\sigma\chi_{\langle x,\vec{e}_{3}\rangle<\sigma}}{\langle x,\vec{e}_{3}\rangle^{2}}\vec{e}_{3},
    \end{equation*}
    and
    \begin{equation*}
        \nabla\rho(x)=\frac{1}{\langle x,\vec{e}_{3}\rangle}\Big(\vec{e}_{3}+o(1)\Big),\quad\nabla\rho(x_{\theta})=\frac{1}{\langle x_{\theta},\vec{e}_{3}\rangle}\Big(\vec{e}_{3}+o(1)\Big),
    \end{equation*}
    where we have used $x,x_{\theta}\in\Sigma_{2\epsilon}$. Here, $o(1)\to0$ as $\epsilon\to0$. This implies
    \begin{equation*}
        \nabla\rho_{\theta}(x)=\nabla\rho(x_{\theta})-2\langle\nabla\rho(x_{\theta}),\vec{v}_{\theta}\rangle\vec{v}_{\theta}=\frac{1}{\langle x_{\theta},\vec{e}_{3}\rangle}\Big(-\vec{e}_{3}+o(1)\Big).
    \end{equation*}
    Consequently, for  $\epsilon$ sufficiently small, we have
    \begin{equation*}
        \nabla w_{\theta,\sigma}=|\nabla w_{\theta,\sigma}|\Big(-\vec{e}_{3}+o(1)\Big),\quad\frac{\nabla\rho_{\theta}}{\sqrt{1+|\nabla\rho_{\theta}|^{2}}}-\frac{\nabla\rho}{\sqrt{1+|\nabla\rho|^{2}}}=-2\vec{e}_{3}+o(1).
    \end{equation*}
    Then \eqref{eq. energy inequality for small theta} implies
    \begin{equation*}
        \int_{\Omega_{\theta,\sigma}}|\nabla w_{\theta,\sigma}|dx\leq C\int_{\Omega_{\theta,\sigma}}|w_{\theta,\sigma}|dx.
    \end{equation*}
    However, as $\Omega_{\theta,\sigma}\subseteq\Sigma_{\theta}\subseteq\Sigma_{\epsilon}$, we have from the Poincar\'e inequality (recall that $w_{\theta,\sigma}=0$ on $\partial\Omega_{\theta,\sigma}$) that
    \begin{equation*}
        \int_{\Omega_{\theta,\sigma}}|\nabla w_{\theta,\sigma}|dx\geq\frac{C^{-1}}{\epsilon}\int_{\Omega_{\theta,\sigma}}|w_{\theta,\sigma}|dx.
    \end{equation*}
    We then see $w_{\theta,\sigma}\equiv0$ almost everywhere in $\Omega_{\theta,\sigma}$, a contradiction.
\end{proof}

{\bf 3. Continuous rotation of $T_{\theta}$ and the rotational symmetry}
\begin{lemma} \label{L2}
    Suppose $w_{\theta}\geq0$ in $\Sigma_{\theta}$ for some $\theta\in(0,\frac{\pi}{2})$, then there exists some $\overline{\epsilon}>0$ (depending on $\theta$), such that   $\forall \epsilon\in(0,\overline{\epsilon})$, we have
\[
w_{\theta+\epsilon}(x)\geq0,   \quad  \forall x\in \Sigma_{\theta+\epsilon}. 
\]
\end{lemma}
\begin{proof}
    Similarly, we will show $w_{\theta+\epsilon,\sigma}\geq0$ in $\Sigma_{\theta+\epsilon}$ for every $\sigma>0$. We first let $\delta$ be an arbitrary number with $\delta\leq\frac{1}{2}\min\{\theta,\frac{\pi}{2}-\theta\}$ and decompose:
    \begin{equation*}
        \Sigma_{\theta+\epsilon}=\Sigma_{\theta-\delta}\cup\Big(\Sigma_{\theta+\epsilon}\setminus\Sigma_{\theta-\delta}\Big).
    \end{equation*}
    
\textbf{Step 1: Positivity in $\Sigma_{\theta-\delta}$.} 
    
In this step, we show that for every $\delta\leq\frac{1}{2}\min\{\theta,\frac{\pi}{2}-\theta\}$, there exists some $\overline{\epsilon}(\delta)\leq\delta$, such that for every $\epsilon\leq\overline{\epsilon}(\delta)$,
    \begin{equation}\label{eq. limsup step 1}
        w_{\theta+\epsilon}(x)>0,\quad\mbox{for all }x\in\Sigma_{\theta-\delta}.
    \end{equation}
    
For every $x\in\Sigma_{\theta-\delta}$, we can assume $x\in T_{\alpha}$ for some $\alpha\in(0,\theta-\delta)$, then for every $\gamma\in(\theta,\theta+\delta)$, we have
    \begin{equation*}
        \frac{\langle x_{\gamma},\vec{e}_{3}\rangle}{\langle x,\vec{e}_{3}\rangle}=\frac{\sin{(2\gamma-\alpha)}}{\sin{\alpha}}=\frac{\sin{(\min\{2\gamma-\alpha,\pi-2\gamma+\alpha\})}}{\sin{\alpha}}.
    \end{equation*}
Notice that when $\delta\leq\frac{1}{2}\min\{\theta,\frac{\pi}{2}-\theta\}$, $\gamma\in(\theta,\theta+\delta)$, and $\alpha\in(0,\theta-\delta)$, we have
    \begin{equation*}
        \alpha+2\delta\leq\min\{2\gamma-\alpha,\pi-2\gamma+\alpha\}\leq\frac{\pi}{2},
    \end{equation*}
    which means that there exists some $0<c_{1}(\delta)\leq\frac{1}{10}$, such that
    \begin{equation*}
        \frac{\langle x_{\gamma},\vec{e}_{3}\rangle}{\langle x,\vec{e}_{3}\rangle}\geq1+c_{1}(\delta)>1\mbox{ in }\Sigma_{\theta-\delta}.
    \end{equation*}
    For such a constant $c_{1}(\delta)$, we can find a corresponding constant $h_{1}(\delta)$, such that
    \begin{equation*}
        \Big|\frac{e^{\rho(y)}}{2\langle y,\vec{e}_{3}\rangle}-1\Big|\leq\frac{c_{1}(\delta)}{100},\quad\mbox{for all }y\in\mathbb{S}^{2}_{+}\mbox{ with }\langle y,\vec{e}_{3}\rangle\leq h_{1}(\delta).
    \end{equation*}
    Then, we also find a constant $h_{2}(\delta)\leq h_{1}(\delta)$, such that
    \begin{equation*}
        \rho(z)\leq\rho(y),\quad\mbox{for all }y,z\in\mathbb{S}^{2}_{+}\mbox{ with }\langle y,\vec{e}_{3}\rangle\geq h_{1}(\delta)\mbox{ and }\langle z,\vec{e}_{3}\rangle\leq h_{2}(\delta).
    \end{equation*}
    This allows us to prove \eqref{eq. limsup step 1} with the additional assumption that $\langle x,\vec{e}_{3}\rangle\leq h_{2}(\delta)$. In fact, we let $\gamma\in[\theta,\theta+\delta]$. If $\langle x_{\gamma},\vec{e}_{3}\rangle\geq h_{1}(\delta)$, then \eqref{eq. limsup step 1} is obvious. On the other hand, if $\langle x_{\gamma},\vec{e}_{3}\rangle\leq h_{1}(\delta)$, we then have
    \begin{align*}
        e^{\rho(x_{\gamma})}-e^{\rho(x)}\geq&2\langle x_{\gamma},\vec{e}_{3}\rangle\Big(1-\frac{c_{1}(\delta)}{100}\Big)-2\langle x,\vec{e}_{3}\rangle\Big(1+\frac{c_{1}(\delta)}{100}\Big)\\
        \geq&2\langle x,\vec{e}_{3}\rangle\Big\{\Big(1+c_{1}(\delta)\Big)\cdot\Big(1-\frac{c_{1}(\delta)}{100}\Big)-\Big(1+\frac{c_{1}(\delta)}{100}\Big)\Big\}>0,
    \end{align*}
    which implies that \eqref{eq. limsup step 1} holds if we additionally know that $\langle x,\vec{e}_{3}\rangle\leq h_{2}(\delta)$.

    For every $x\in\Sigma_{\theta-\delta}$ with $\langle x,\vec{e}_{3}\rangle\geq h_{2}(\delta)$, we notice that $\langle x_{\gamma},\vec{e}_{3}\rangle\geq h_{3}(\theta,\delta)>0$ for all $\gamma\in[\theta,\theta+\delta]$, indicating that
    \begin{equation}\label{eq. gradient bound in limsup}
        \rho(x),\rho_{\gamma}(x),|\nabla\rho(x)|,|\nabla\rho_{\gamma}(x)|\leq C_{2}(\theta,\delta)
    \end{equation}
    Then, the proof of \eqref{eq. limsup step 1} in this case is reduced to showing the following strict inequality for some $c_{3}(\theta,\delta)>0$:
    \begin{equation}\label{eq. limsup strong maximum principle}
        w_{\theta}(x)\geq c_{3}(\theta,\delta)>0,\quad\mbox{for all }x\in\overline{\Sigma_{\theta-\delta}}\mbox{ with }\langle x,\vec{e}_{3}\rangle\geq h_{2}(\delta).
    \end{equation}
Since $\overline{\Sigma_{\theta-\delta}}\cap\{x:\langle x,\vec{e}_{3}\rangle\geq h_{2}(\delta)\}$ is a compact set, we can show \eqref{eq. limsup strong maximum principle} by contradiction. 

Suppose that \eqref{eq. limsup strong maximum principle} fails, then by $w_{\theta}\geq0$ in $\Sigma_{\theta}$, we can assume that $w_{\theta}(x)=0$ for some $x\in\overline{\Sigma_{\theta-\delta}}\cap\{x:\langle x,\vec{e}_{3}\rangle\geq h_{2}(\delta)\}$. For $y\in\mathbb{S}^{2}_{+}$ sufficiently close to $x$, we have that $\rho$ and $\rho_{\theta}$ both satisfy the equation \eqref{eq: radial}. Therefore,
    \begin{equation}\label{eq. difference equation}
        \mathrm{div}\Big(\frac{\nabla\rho_{\theta}}{\sqrt{1+|\nabla\rho_{\theta}|^2}}-\frac{\nabla\rho}{\sqrt{1+|\nabla\rho|^2}}\Big)=\frac{1}{2}e^{2\rho}\sqrt{1+|\nabla\rho|^2}-\frac{1}{2}e^{2\rho_{\theta}}\sqrt{1+|\nabla\rho_{\theta}|^2}.
    \end{equation}
    For the left-hand side of \eqref{eq. difference equation}, we apply \eqref{eq. gradient bound in limsup}, Lemma~\ref{lem. uniform elliptic in energy sense} and Lemma~\ref{lem. linear algebra}, we see
    \begin{equation*}
        \frac{\nabla\rho_{\theta}}{\sqrt{1+|\nabla\rho_{\theta}|^2}}-\frac{\nabla\rho}{\sqrt{1+|\nabla\rho|^2}}=A\cdot(\nabla\rho_{\theta}-\nabla\rho)=A\cdot\nabla w_{\theta},
    \end{equation*}
    where $A$ is a self-adjoint uniformly elliptic operator near $x$. For the right-hand side of \eqref{eq. difference equation}, we use the boundedness estimate \eqref{eq. gradient bound in limsup} near $x$ to conclude that
    \begin{align*}
        \frac{1}{2}e^{2\rho}\sqrt{1+|\nabla\rho|^2}-\frac{1}{2}e^{2\rho_{\theta}}\sqrt{1+|\nabla\rho_{\theta}|^2}=\vec{b}\cdot\nabla w_{\theta}+c\cdot w_{\theta},
    \end{align*}
    where the coefficients $\vec{b}$ and $c$ are uniformly bounded near $x$. Thus \eqref{eq. difference equation} indicates that
    \begin{equation*}
        \mathrm{div}(A\cdot\nabla w_{\theta})=\vec{b}\cdot\nabla w_{\theta}+c\cdot w_{\theta}.
    \end{equation*}
    By the strong maximum principle, we see $w_{\theta}\equiv0$ near $x$. By the open-and-closed argument, we have $w_{\theta}\equiv0$ everywhere in $\Sigma_{\theta}$. This contradicts the asymptotic behavior that $w_{\theta}(x)=+\infty$ as $x\to x^{*}\in\Sigma_{\theta}\cap\{\langle x,\vec{e}_{3}\rangle=0\}$. Then, we have shown \eqref{eq. limsup strong maximum principle} and hence \eqref{eq. limsup step 1}.

\textbf{Step 2: Positivity in $\Sigma_{\theta+\epsilon}\setminus\Sigma_{\theta-\delta}$.} 
    
In this step, we show that there exists some $\delta\leq\frac{1}{2}\min\{\theta,\frac{\pi}{2}-\theta\}$, such that for every $\epsilon\leq\overline{\epsilon}(\delta)$, $w_{\theta+\epsilon}(x)\geq0$ in $\Sigma_{\theta+\epsilon}\setminus\Sigma_{\theta-\delta}$. In fact, it suffices to show the following inequality for $\epsilon\leq\overline{\epsilon}(\delta)$:
    \begin{equation}
        w_{\theta+\epsilon,\sigma}(x)\geq0,\quad\mbox{for all }x\in\Sigma_{\theta+\epsilon}\setminus\Sigma_{\theta-\delta}\mbox{ and all }\sigma>0.
    \end{equation}
    We multiply by $w_{\theta+\epsilon,\sigma}$ on both sides of the following equation
    \begin{equation*}
        \mathrm{div}(\frac{\nabla\rho}{\sqrt{1+|\nabla\rho|^{2}}})-\mathrm{div}(\frac{\nabla\rho_{\theta+\epsilon}}{\sqrt{1+|\nabla\rho_{\theta+\epsilon}|^{2}}})=\frac{e^{2\rho_{\theta+\epsilon}}}{2}\sqrt{1+|\nabla\rho_{\theta}|^{2}}-\frac{e^{2\rho}}{2}\sqrt{1+|\nabla\rho|^{2}},
    \end{equation*}
    and integrate in $\Omega_{\theta+\epsilon,\sigma}$. By Lemma~\ref{lem. strong inclusion} and Step 1, we know $\Omega_{\theta+\epsilon,\sigma}\subseteq\Big(\Sigma_{\theta+\epsilon}\setminus\Sigma_{\theta-\delta}\Big)$, and no boundary term is produced, so we have
    \begin{align}\label{eq. energy identity in step 2}
        &\int_{\Omega_{\theta+\epsilon,\sigma}}\langle\nabla w_{\theta+\epsilon,\sigma},\frac{\nabla\rho_{\theta+\epsilon}}{\sqrt{1+|\nabla\rho_{\theta+\epsilon}|^{2}}}-\frac{\nabla\rho}{\sqrt{1+|\nabla\rho|^{2}}}\rangle dx\\
        =&\int_{\Omega_{\theta+\epsilon,\sigma}}w_{\theta+\epsilon,\sigma}\cdot\Big(\frac{e^{2\rho_{\theta+\epsilon}}}{2}\sqrt{1+|\nabla\rho_{\theta+\epsilon}|^{2}}-\frac{e^{2\rho}}{2}\sqrt{1+|\nabla\rho|^{2}}\Big)dx.
    \end{align}

To analyze this equality, we decompose $\Omega_{\theta+\epsilon,\sigma}$ into two parts.
    
First, we let $h_{4}(\theta,\delta)$ be sufficiently small, such that for every $\epsilon\leq\delta$, and for every $x\in\Omega_{\theta+\epsilon,\sigma}\subseteq\Sigma_{\theta+\epsilon}\setminus\Sigma_{\theta-\delta}$ satisfying $\langle x,\vec{e}_{3}\rangle\leq h_{4}(\theta,\delta)$, $\langle x_{\theta+\epsilon},\vec{e}_{3}\rangle$ is also small. Based on \eqref{asym1} and \eqref{asym2}, we have
    \begin{equation*}
        \nabla\rho(x)=\frac{1}{\langle x,\vec{e}_{3}\rangle}\Big(\vec{e}_{3}+o(1)\Big),\quad\nabla\rho(x_{\theta+\epsilon})=\frac{1}{\langle x_{\theta+\epsilon},\vec{e}_{3}\rangle}\Big(\vec{e}_{3}+o(1)\Big).
    \end{equation*}
    Here, $o(1)$ is sufficiently small when $\delta$ is sufficiently small and $\epsilon\leq\delta$. Notice that for sufficiently small $\delta$ (and $\epsilon\leq\delta$), $\langle x_{\theta+\epsilon},\vec{e}_{3}\rangle=\langle x,\vec{e}_{3}\rangle\cdot\Big(1+o(1)\Big)$ in $\Sigma_{\theta+\epsilon}\setminus\Sigma_{\theta-\delta}$, and $\vec{v}_{\theta+\epsilon}=\vec{v}_{\theta}+o(1)$ (see the notation of $\vec{v}_{\theta}$ in Definition~\ref{def. moving plane setting}), so we have
    \begin{equation*}
        \nabla\rho_{\theta+\epsilon}(x)=\nabla\rho(x_{\theta+\epsilon})-2\langle\nabla\rho(x_{\theta+\epsilon}),\vec{v}_{\theta+\epsilon}\rangle\vec{v}_{\theta+\epsilon}=\frac{1}{\langle x,\vec{e}_{3}\rangle}\Big(\sin{(2\theta)}\vec{e}_{1}-\cos{(2\theta)}\vec{e}_{3}+o(1)\Big).
    \end{equation*}
    Therefore, we have that
    \begin{equation*}
        \frac{\nabla\rho_{\theta+\epsilon}}{\sqrt{1+|\nabla\rho_{\theta+\epsilon}|^{2}}}-\frac{\nabla\rho}{\sqrt{1+|\nabla\rho|^{2}}}=\sin{(2\theta)}\vec{e}_{1}-2\cos^{2}{(\theta)}\vec{e}_{3}+o(1).
    \end{equation*}
    For a given $\theta$, we then see that there exists a positive constant $\lambda_{1}=\lambda_{1}(\theta)$, such that for all $x\in\Omega_{\theta+\epsilon,\sigma}$ with $\langle x,\vec{e}_{3}\rangle\leq h_{4}(\theta,\delta)$, it holds
    \begin{equation*}
        \langle\nabla w_{\theta+\epsilon,\sigma},\frac{\nabla\rho_{\theta+\epsilon}}{\sqrt{1+|\nabla\rho_{\theta}|^{2}}}-\frac{\nabla\rho}{\sqrt{1+|\nabla\rho|^{2}}}\rangle\geq\lambda_{1}|\nabla w_{\theta+\epsilon,\sigma}|.
    \end{equation*}
    Moreover, in this region, we have
    \begin{equation*}
        w_{\theta+\epsilon,\sigma}\cdot\Big(\frac{e^{2\rho_{\theta+\epsilon}}}{2}\sqrt{1+|\nabla\rho_{\theta+\epsilon}|^{2}}-\frac{e^{2\rho}}{2}\sqrt{1+|\nabla\rho|^{2}}\Big)\leq C|w_{\theta+\epsilon,\sigma}|.
    \end{equation*}

    Second, we consider $x\in\Omega_{\theta+\epsilon,\sigma}$ with $\langle x,\vec{e}_{3}\rangle\geq h_{4}(\theta,\delta)$. In this region, both $\langle x,\vec{e}_{3}\rangle$ and $\langle x_{\theta+\epsilon},\vec{e}_{3}\rangle$ are bounded from below, implying the boundedness of $|\nabla\rho(x)|$ and $|\nabla\rho_{\theta+\epsilon}(x)|$. Moreover, by choosing $\sigma$ small, depending on $h_{4}(\theta,\delta)$, we can even ensure that
    \begin{equation*}
        w_{\theta+\epsilon,\sigma}\equiv w_{\theta+\epsilon}\mbox{ when }\langle x,\vec{e}_{3}\rangle\geq h_{4}(\theta,\delta).
    \end{equation*}
    By Lemma~\ref{lem. uniform elliptic in energy sense}, we have
    \begin{align*}
        \langle\nabla w_{\theta+\epsilon,\sigma},\frac{\nabla\rho_{\theta}}{\sqrt{1+|\nabla\rho_{\theta}|^{2}}}-\frac{\nabla\rho}{\sqrt{1+|\nabla\rho|^{2}}}\rangle=&\langle\nabla\rho_{\theta+\epsilon}-\nabla\rho,\frac{\nabla\rho_{\theta}}{\sqrt{1+|\nabla\rho_{\theta}|^{2}}}-\frac{\nabla\rho}{\sqrt{1+|\nabla\rho|^{2}}}\rangle\\
        \geq&\lambda_{2}|\nabla\rho_{\theta+\epsilon}-\nabla\rho|^{2}=\lambda_{2}|\nabla w_{\theta+\epsilon,\sigma}|^{2}.
    \end{align*}
    Besides, for $x\in\Omega_{\theta+\epsilon,\sigma}$ with $\langle x,\vec{e}_{3}\rangle\geq h_{4}(\theta,\delta)$, we have
    \begin{align*}
        \frac{e^{2\rho_{\theta+\epsilon}}}{2}\sqrt{1+|\nabla\rho_{\theta+\epsilon}|^{2}}-\frac{e^{2\rho}}{2}\sqrt{1+|\nabla\rho|^{2}}=&\vec{b}\cdot(\nabla\rho_{\theta+\epsilon}-\nabla\rho)+c\cdot(\rho_{\theta+\epsilon}-\rho)\\
        =&\vec{b}\cdot\nabla w_{\theta+\epsilon,\sigma}+c\cdot w_{\theta+\epsilon,\sigma}.
    \end{align*}
    Then, we have:
    \begin{equation*}
        w_{\theta+\epsilon,\sigma}\cdot\Big(\frac{e^{2\rho_{\theta+\epsilon}}}{2}\sqrt{1+|\nabla\rho_{\theta+\epsilon}|^{2}}-\frac{e^{2\rho}}{2}\sqrt{1+|\nabla\rho|^{2}}\Big)\leq\frac{\lambda_{2}}{2}|\nabla w_{\theta+\epsilon,\sigma}|^{2}+C|w_{\theta+\epsilon,\sigma}|^{2}.
    \end{equation*}
    We then obtain from the identity \eqref{eq. energy identity in step 2} the following estimate:
    \begin{align*}
        &\int_{\Omega_{\theta+\epsilon,\sigma}\cap\{\langle x,\vec{e}_{3}\rangle\leq h_{4}\}}|\nabla w_{\theta+\epsilon,\sigma}|dx+\int_{\Omega_{\theta+\epsilon,\sigma}\cap\{\langle x,\vec{e}_{3}\rangle\geq h_{4}\}}|\nabla w_{\theta+\epsilon,\sigma}|^{2}dx\\
        \leq&C\int_{\Omega_{\theta+\epsilon,\sigma}\cap\{\langle x,\vec{e}_{3}\rangle\leq h_{4}\}}|w_{\theta+\epsilon,\sigma}|dx+C\int_{\Omega_{\theta+\epsilon,\sigma}\cap\{\langle x,\vec{e}_{3}\rangle\geq h_{4}\}}|w_{\theta+\epsilon,\sigma}|^{2}dx.
    \end{align*}
    However, recall that $w_{\theta+\epsilon,\sigma}=0$ on $\partial\Omega_{\theta+\epsilon,\sigma}$ and that
    \begin{equation*}
        \Omega_{\theta+\epsilon,\sigma}\subseteq\Sigma_{\theta+\epsilon}\setminus\Sigma_{\theta-\delta}\subseteq\Sigma_{\theta+\delta}\setminus\Sigma_{\theta-\delta},
    \end{equation*}
    we then infer from the Poincar\'e inequality that
    \begin{align*}
        \int_{\Omega_{\theta+\epsilon,\sigma}\cap\{\langle x,\vec{e}_{3}\rangle\leq h_{4}\}}|\nabla w_{\theta+\epsilon,\sigma}|dx\geq&\frac{C^{-1}}{\delta}\int_{\Omega_{\theta+\epsilon,\sigma}\cap\{\langle x,\vec{e}_{3}\rangle\leq h_{4}\}}|w_{\theta+\epsilon,\sigma}|dx,\\
        \int_{\Omega_{\theta+\epsilon,\sigma}\cap\{\langle x,\vec{e}_{3}\rangle\geq h_{4}\}}|\nabla w_{\theta+\epsilon,\sigma}|^{2}dx\geq&\frac{C^{-1}}{\delta^{2}}\int_{\Omega_{\theta+\epsilon,\sigma}\cap\{\langle x,\vec{e}_{3}\rangle\geq h_{4}\}}|w_{\theta+\epsilon,\sigma}|^{2}dx.
    \end{align*}
    Therefore, for sufficiently small $\delta$, we must have that $w_{\theta+\epsilon,\sigma}\equiv0$ in $\Omega_{\theta+\epsilon,\sigma}$, and thus $\Omega_{\theta+\epsilon,\sigma}=\emptyset$. Then we have shown the positivity of $w_{\theta+\epsilon}$ in the region $\Sigma_{\theta+\epsilon}$.
\end{proof}
So far we have shown that $\rho:\mathbb{S}^{2}_{+}\to\mathbb{R}$ is symmetric about the plane $\{x_{1}=0\}$. The same argument above is valid for any other directions in $span(\vec{e}_{1},\vec{e_{2}})$, and hence $\rho(x)$ is rotational symmetric about the $\vec{e}_{3}$ axis.

{\bf 4. Uniqueness under rotational symmetry}

We have already shown that $\rho:\mathbb{S}^{2}_{+}\to\mathbb{R}$ is a function depending only on $\theta=\arcsin{\langle x,\vec{e}_{3}\rangle}$. Set 
\begin{equation*}
    u(\theta)=e^{\rho(\theta)},\quad\theta\in[0,\frac{\pi}{2}],
\end{equation*}
it follows from \eqref{eq: radial} that $u:[0,\frac{\pi}{2}]\to\mathbb{R}_{+}$ satisfies the following ODE: 
\begin{equation}\label{eq. ODE}
    u^{2}\cdot u_{\theta\theta}-u\cdot u_{\theta}^{2}-\tan{\theta}\cdot u_{\theta}(u^{2}+u_{\theta}^{2})+\frac{u}{2}(u^{2}+u_{\theta}^{2})^{2}=0.
\end{equation}

We shall show that $v(\theta)=2\sin(\theta)$ is the only solution to \eqref{eq. ODE}. First note that if $M$ is a radial graph over $\mathbb{S}^2_{+}$ with radial function $\rho(\theta)=\ln(2 \sin(\theta))$, then $M$ is just a sphere of radius $1$ centered at $\vec{e}_{3}$.

Suppose there admits {\bf another} radial solution $u(\theta)$ of \eqref{eq: radial}, denote by $N$ its corresponding surface. 

We {\bf Claim} that $N$ stays above $M$ near the origin. We call the plane spanned by $\vec{e}_1$ and $\vec{e}_2$ the $y_{1}-y_{2}$ plane. We can view both $N$ and $M$ as graphs of two radial symmetric functions defined on a neighborhood of the origin in the $y_{1}-y_{2}$ plane. Denote these functions by $\varphi_{N}$ and $\varphi_{M}$ respectively. 
We have 
\[
H_N=\operatorname{div_{y_{1},y_{2}}}(\frac{\nabla \varphi_N}{\sqrt{1+|\nabla \varphi_N|^2}})\geq 2=H_M=\operatorname{div_{y_{1},y_{2}}}(\frac{\nabla \varphi_M}{\sqrt{1+|\nabla \varphi_M|^2}}).
\]
The second inequality follows from \eqref{eq:surface2}.
If there exists $r_0>0$ such that $\varphi_{N}(r_0)<\varphi_M(r_0)$, then by the maximum principle for the usual minimal surface equation, we derive that $\varphi_N(0)<\varphi_M(0)=0$, a contradiction. 

Since $N$ stays above $M$ near the origin, there exists $\theta_0>0$ such that  $u(\theta)\leq v(\theta), \forall \theta\in[0, \theta_0)$ and $u(\theta_0)<v(\theta_0)$.

Let
\begin{equation}\label{eq. fully nonlinear operator F()}
    F(M, P, Z, \theta)=M-\frac{P^2}{Z}-\tan(\theta)\frac{P(Z^2+P^2)}{Z^2}+\frac{(Z^2+P^2)^2}{2Z}.
\end{equation}
In view of \eqref{eq. ODE}, $F(u'', u', u, \theta)=F(v'', v', v, \theta)$ for all $\theta \in (0, \theta_0)$. Setting $w=u-v$, then by the Lagrange mean value theorem, we have 
\begin{equation}\label{eq. Lagrange MVT}
F_{M}(\xi_{\theta}) w''(\theta)+ F_{P}(\xi_{\theta}) w'(\theta)+ F_Z(\xi_{\theta}) w(\theta)=0,
\end{equation}
where $\xi_{\theta}$ is a point on the line segment from $(u''(\theta), u'(\theta), u(\theta), \theta)$ to $(v''(\theta), v'(\theta), v(\theta), \theta)$.

In view of the asymptotic behavior \eqref{asym1} and \eqref{asym2}, we have 
\[
(\xi_{\theta})=(\star, 2+o(1), 2\theta+o(\theta), \theta).
\]
Although the value of $\star$ is unclear, we still get from \eqref{eq. fully nonlinear operator F()} and \eqref{eq. Lagrange MVT} that $w(\theta)$ satisfies
\[
w''(\theta)+  (\frac{3+o(1)}{\theta})w'(\theta)+ (\frac{1+o(1)}{\theta^2})w(\theta)=0.
\]
Since $w(\theta_0)<0$, there exists $\kappa>0$ such that 
$w(\theta_0)+\kappa \theta_0\leq 0.$
The function $\psi(\theta):=w(\theta)+\kappa \theta$ satisfies 
\[
\psi''(\theta)+(\frac{3+o(1)}{\theta})(\psi'(\theta)-\kappa)+(\frac{1+o(1)}{\theta^2})(\psi-\kappa \theta)=0.
\]
A simple maximum argument shows that $\psi$ cannot have any positive maximum in the interior of $(0, \theta_0)$. Thus $\psi(\theta)\leq 0$, which implies that $w'(0)\leq -\kappa<0$. This contradicts to $w'(0)=0$.

Above argument indicates that $N$ coincides with $M$ near the origin, but once we are away from the origin, the standard Picard's theorem for the uniqueness of \eqref{eq. ODE} applies and the proof thus is finished. 
\end{proof}
\section{Further discussion}
We conclude this paper with two remarks. As indicated in the introduction, the strong Green function rigidity for the conformal Laplace is proved via the connection with positive mass theorem. An exploration along this direction together with our main theorem indicates the following theorem. 

\begin{theorem}
    Let $M\subset \mathbb{R}^3$ be a closed embedded surface with $0\in M$ and let $K(y)$ be its Gaussian curvature. Let $\tilde{M}$ denote the image of $M$ under the Kelvin transform $\Phi(y)=\frac{y}{|y|^2}$. Suppose that the Gaussian curvature $\widetilde{K}(y)$ of $\widetilde{M}$ has the following correspondence with $K(y)$:
    \[
    \widetilde{K}(y)=\frac{1}{|y|^4}\left(K(\frac{y}{|y|^2})-1\right), 
    \]
    then $M$ must be a round sphere. 
\end{theorem}

\begin{proof}
Let $\widetilde{g}$ be the restriction of Euclidean metric on $\widetilde{M}$.
    A direct computation shows that 
    $\widetilde{g}(\Phi(y))=\frac{1}{|y|^4} g(y)$, i.e,. $\widetilde{g}$ can be viewed as a metric conformal to $g$ on $M\setminus\{0\}$. The equation of the Gauss curvature under conformal change implies that
    \[
    \Delta_g( \ln(\frac{1}{|y|^2})) + \widetilde{K}(\Phi(y)) \frac{1}{|y|^4}= K,
    \]
    from which we find the Green function of $M$ at the origin has the form $G(y):=-\ln|y|$. Our main theorem yields the desired conclusion.    
\end{proof}

We hope our method can also be used to deal with higher dimensional cases. A derivation similar to \eqref{eq:laplaceonM},\eqref{eq:lapalce},\eqref{eq:surface1}
shows the following
\begin{proposition}
    Let $M^n \subset\mathbb{R}^{n+1}, (n\geq 3)$ be a closed embedded hypersurface with $0\in M$.
    Then the Green function for its conformal Laplace $(-\Delta_g+\frac{(n-2)}{4(n-1)}R_g)$ has the form $G(y,0)=c_n |y|^{2-n}$ if and only if the mean curvature  $H$ (the averaged version $H=\frac{\kappa_1+\cdots +\kappa_n}{n}$) satisfies
    \begin{align} \label{eq: conformal}
    H=\frac{\langle y, \nu \rangle}{|y|^2}+\frac{R}{4n(n-1)} \frac{|y|^2}{\langle y, \nu\rangle}, \quad y\in M\setminus\{0\}
    \end{align}
  \end{proposition}

Note here $R$ denotes the scalar curvature of $M$, which can be also expressed as $2\sigma_2(II)$. Therefore \eqref{eq: conformal} seems to be more complicated to handle within our approach.

\end{document}